\documentclass{amsart}
\usepackage{fullpage}

\usepackage{amssymb}

\usepackage{hyperref}
\usepackage{amsthm}
\usepackage{cleveref}

\theoremstyle{definition}
\newtheorem{theorem}{Theorem}[section]

\newtheorem{proposition}[theorem]{Proposition}
\newtheorem{corollary}[theorem]{Corollary}
\newtheorem{lemma}[theorem]{Lemma}
\newtheorem{definition}[theorem]{Definition}

\newtheorem{remark}[theorem]{Remark}

\numberwithin{equation}{section}

\usepackage{mathtools}

\newcommand{\Evol}{\textsf{Evol}} % Euclidean Volume
\newcommand{\nvol}{\textsf{nvol}} % Normalized Volume
\newcommand{\PS}{\operatorname{PS}} % Pitman Stanley Polytope
\newcommand{\PF}{\operatorname{PF}} % Parking Functions
\newcommand{\PC}{\operatorname{PC}} % Parking Completions
\usepackage{listings}
\usepackage{xcolor}
\definecolor{codegreen}{rgb}{0,0.6,0}
\definecolor{codegray}{rgb}{0.5,0.5,0.5}
\definecolor{codepurple}{rgb}{0.58,0,0.82}
\definecolor{backcolour}{rgb}{0.95,0.95,0.92}

\lstdefinestyle{mystyle}{
    backgroundcolor=\color{backcolour},   
    commentstyle=\color{codegreen},
    keywordstyle=\color{magenta},
    numberstyle=\tiny\color{codegray},
    stringstyle=\color{codepurple},
    basicstyle=\ttfamily\footnotesize,
    breakatwhitespace=false,         
    breaklines=true,                 
    captionpos=b,                    
    keepspaces=true,                 
    numbers=left,                    
    numbersep=5pt,                  
    showspaces=false,                
    showstringspaces=false,
    showtabs=false,                  
    tabsize=2
}

\begin{document}

\title{Parking completions are $\mathbf{x}$-parking functions}


%    author one information
\author{Kimberly P. Hadaway}
\address{Iowa State University, Carver Hall, 411 Morrill Road, Ames, IA 50014}
\curraddr{}
\email{kph3@iastate.edu}
\thanks{}

% %    author two information
% \author{}
% \address{}
% \curraddr{}
% \email{}
% \thanks{}

%    \subjclass is required.

\date{\today}

\dedicatory{}

%    Abstract is required.
\begin{abstract}
Parking functions correspond with preferences of $n$ cars which enter sequentially to park on a one-way street where (1) each car parks in the first available spot greater than or equal to its preference and (2) all cars successfully park. 
We generalize parking functions to parking completions: 
Here, we are given that some cars have already parked in a set of spots, which are indexed in a sequence~$\mathbf{t}$. 
We then consider a preference list $\mathbf{c}$, where length of $\mathbf{t}$ + length of $\mathbf{c}$ = $n$. 
If all cars can park, we say that $\mathbf{c}$ is a parking completion. 
Adeniran et al.\ (2020) state an open problem which proposes a connection between the number of parking completions to the volumes of Pitman-Stanley polytopes by explicit computation on small values of $n$. 
In this paper, we provide a solution to this open problem by proving a theorem which explains that the set of parking completions is the set of $\mathbf{x}$-parking functions.
\end{abstract}
\maketitle


% \section{Introduction}\label{sec:intro_short_paper}

\section{Introduction}

Our main result establishes that parking completions are $\mathbf{x}$-parking functions. 
To this end, we provide some context to these families of combinatorial objects and a connection to Pitman-Stanley polytopes.
In this section, we detail definitions and known results about parking functions, their generalizations, polytopes, and a specific family called Pitman-Stanley polytopes.

\subsection{Parking Functions}\label{sec:pf_background_short_paper}
The parking function problem was originally proposed by Konheim and Weiss \cite{bib:KW} in 1962. 
Since then, a less gendered version of the problem has been studied broadly.
We begin with defining (classical) parking functions.
Imagine a one-way one-lane dead-end street with $n$ parking spots numbered $1, 2, \hdots, n$ in increasing order and $n$ cars waiting to park.
Let $\mathbf{a} = (a_1, a_2, \hdots, a_n) \in [n]^{n}$ where $[n] \coloneqq \{1, 2, \hdots, n\}$; we call $\mathbf{a}$ a \textbf{preference list}.
For each $i\in [n]$, we interpret $a_i$ as the parking preference for car $c_i$.
The cars enter the street in order $c_1, c_2, \hdots, c_n$ and park according to this rule: 
When car $c_i$ enters the street, it approaches spot $a_i$.
If spot $a_i$ is open, car $c_i$ parks there.
If not, car $c_i$ proceeds down the street and takes the next available spot (if one exists).
If all of the cars can park with preferences as described in $\mathbf{a}$, then we call $\mathbf{a}$ a \textbf{parking function} (of length $n$).

Parking functions have been enumerated by N. Shales \cite{bib:recursionproof} using a recursive formula and by Henry O.\ Pollak \cite{bib:stanleydocpollakproof} using an explicit formula; both formulas have been further explained in expository work by Hadaway \cite{bib:GirlsAnglePart1,bib:GirlsAnglePart2}.
Moving forward, we write $\PF_n$ for the set of all parking functions of length $n$.
An equivalent definition of a parking function follows:
    Let $\mathbf{a}$ be a preference list.
    Let $\mathbf{a}^\uparrow$ denote the weakly increasing rearrangement of~$\mathbf{a}$. Then,
    \begin{align}
    \mbox{$\mathbf{a}$ is a parking function if and only if $\mathbf{a}_i^\uparrow \leq i$ for all $i\in[n]$.}\label{prop:pf_increasing_condition}
    \end{align}
This immediately implies 
the following: 
\begin{align}\label{cor:any_pf_permutation_works}
    \mbox{Every permutation of a parking function $\mathbf{a}$ is a parking function.}
\end{align}
Statements~\eqref{prop:pf_increasing_condition} and~\eqref{cor:any_pf_permutation_works} are well-referenced in the literature on parking functions.

We now discuss two main generalizations of parking functions: $\mathbf{x}$-parking functions and parking completions.
Generalizing Statement~\eqref{prop:pf_increasing_condition}, we change the upper bound in the inequality for each entry in our preference lists, and we get the following definition.

\begin{definition}\label{def:x_parking_function}
    Let $\mathbb{N}=\{1,2,3,\ldots\}$. For $\mathbf{x} \coloneqq (x_1, x_2, \hdots, x_n) \in \mathbb{N}^n$, an \textbf{$\mathbf{x}$-parking function} is a sequence $\mathbf{a} \in\mathbb{N}^n$ whose increasing rearrangement $\mathbf{a}^\uparrow$ satisfies $\mathbf{a}_i^\uparrow \leq x_1 + x_2 + \cdots + x_i$, for all $i\in[n]$.
\end{definition} 

To obtain parking completions from parking functions, we return to a parking scenario.
Imagine that $m$ of the $n$ spots in our original parking setup are taken before the remaining $n-m$ cars begin to park.
We index the taken spots in increasing order in a sequence $\mathbf{t}$, and we have the following definition adapted from Adeniran, Butler, Dorpalen-Barry, Harris, Hettle, Liang, Martin, and Nam \cite{bib:ParkingCompletions}.

\begin{definition}\label{def:parking_completion}
    Fix an increasing list $\mathbf{t} = (t_1, t_2, \hdots, t_m) \in [n]^m$.
    Let $\mathbf{c} = (c_1, c_2, \hdots, c_{n-m}) \in [n]^{n-m}$.
    We let $\mathbf{t}|\mathbf{c} := (t_1, t_2, \hdots, t_m, c_1, c_2, \hdots, c_{n-m})$, and we call this the \textbf{concatenation of $\mathbf{t}$ and $\mathbf{c}$}.
    If $\mathbf{t}|\mathbf{c}$ is a parking function, we say that $\mathbf{c}$ is a \textbf{parking completion for a sequence $\mathbf{t}$}.
    % If $(t_1, t_2, \hdots, t_m, c_1, c_2, \hdots, c_{n-m})$ is a parking function, we say that $\mathbf{c}$ is a \textbf{parking completion for a sequence $\mathbf{t}$}.
    % We then consider a preference list $c$, where $\operatorname{len}(t)+\operatorname{len}(c)=n$. 
    % If all cars can park, we say that $c$ is a parking completion. 
    % \kph{add some exposition here about what a parking completion is and notation that you use for it, specifically, what are $\mathbf{t}$, $\mathbf{c}$, and what does it mean for a vector to be a parking completion.}
\end{definition}

We remark that Adeniran et al.\ \cite{bib:ParkingCompletions} mention in their definition of parking completions that one can equivalently require $c_i \leq u_i$ where $\mathbf{u} = (u_1, u_2, \hdots, u_{n-m})$ denotes the unoccupied spots on the street. 
In \Cref{sec:main_problem_main_result_short_paper}, we include a proof of this equivalence which is formally stated in Lemma~\ref{lem:pc_increasing_condition}. 
This equivalence readily implies that parking completions are permutation invariant which is proven independently in Lemma~\ref{lem:any_pc_permutation_works}.

% \begin{corollary}
%     Every permutation of a parking completion is a parking completion.
% \end{corollary}

\subsection{Polytopes}\label{sec:polytope_background_short_paper}
We now introduce background for Pitman-Stanley polytopes.
Recall for $n \in \mathbb{Z}_{>0}$, a \textbf{polytope} $P \subseteq \mathbb{R}^n$ is the convex hull of finitely many vertices $v_i \in \mathbb{R}^n$ or, equivalently, the intersection of finitely many hyperplanes in $\mathbb{R}^n$. 
In this context, given a polytope $P$, the \textbf{(Euclidean) volume of $P$}, denoted $\Evol$, is its volume computed in the traditional sense.
Moreover, the \textbf{volume of $P$}, denoted $\nvol$, is the normalized volume of $P$ which is computed as $\operatorname{vol}(P) = \operatorname{dim(P)} \cdot ~ \Evol(P)$. 

To define the polytopes called \textit{Pitman-Stanley polytopes}, we recall that 
given 
     $\mathbf{x} = (x_1, x_2, \hdots, x_n)$ and $\mathbf{y} = (y_1, y_2, \hdots, y_n)$ in $\mathbb{R}_{\geq 0}^n$, 
    we say \textbf{$\mathbf{x}$ is less than $\mathbf{y}$ in dominance order} and write $\mathbf{x} \preceq \mathbf{y}$ if 
    \begin{align}
        \mbox{$x_1 + x_2 + \cdots + x_i \leq y_1 + y_2 + \cdots + y_i$ for all $i\in[n]$.}
    \end{align}
We can now define Pitman-Stanley polytopes.
\begin{definition}
    Given a list $\mathbf{a} = (a_1, a_2, \hdots, a_n) \in \mathbb{Z}_{\geq 0}^n$, the \textbf{Pitman-Stanley polytope} (with respect to~$\mathbf{a}$) is 
    \[
    \PS_n(\mathbf{a}) \coloneqq \{ \mathbf{x} = (x_1, x_2, \hdots, x_n) \in \mathbb{R}_{\geq 0}^n \; | \; (x_1, x_2, \hdots, x_n) \preceq (a_1, a_2, \hdots, a_n)\}.
    \]
\end{definition}

It is important to note that the Pitman-Stanley polytope is defined in reference to a fixed list.
That is, if $\mathbf{a}$ and $\mathbf{a'}$ are distinct lists in $\mathbb{Z}_{\geq 0}^n$, then $\PS_n(\mathbf{a}) \neq \PS_n(\mathbf{a'})$ since vectors in the respective sets satisfy different sets of inequalities.

Stanley and Pitman\footnote{The paper \cite{bib:PitmanStanley} lists the authors as ``Stanley and Pitman'' which is unusual for the mathematical sciences as that is not the standard alphabetical order.}
\cite{bib:PitmanStanley} give an explicit formula to enumerate parking functions using volumes of $n$-dimensional polytopes.
To state this result, we recall the following definitions and notation used in \cite[p.\ 1-2]{bib:PitmanStanley}:
    An \textbf{$n$-dimensional polytope} is
    % \[
    % \Pi_n(\mathbf{x}) \coloneqq \{ \mathbf{y} \in \mathbb{R}^n : y_i \geq 0 \mbox{ and } y_1 + y_2 + \cdots + y_i \leq x_1 + x_2 + \cdots + x_i \mbox{ for all $i \in [n]$} \}
    % \]
    \[
    \Pi_n(\mathbf{x}) \coloneqq \{ \mathbf{y} \in \mathbb{R}^n : y_i \geq 0 \mbox{ and } \mathbf{y} \preceq \mathbf{x} \}
    \]
    for arbitrary $\mathbf{x} \coloneqq (x_1, x_2, \hdots, x_n)$ with $x_i > 0$ for all $i \in [n]$.
    The $n$-dimensional volume of this polytope
    \[
    V_n(\mathbf{x}) \coloneqq \operatorname{Vol}(\Pi_n(\mathbf{x}))
    \]
    is a homogeneous polynomial of degree $n$ in the variables $x_1, x_2, \hdots, x_n$ which we call the \textbf{volume polynomial}.
Notice that $V_n(\mathbf{x}) = \Evol(\PS_n(\mathbf{x}))$.
% , and we interchange notation,
% using $V_n(\mathbf{x})$ to reference content from \cite{bib:PitmanStanley} and $\Evol(\PS_n(\mathbf{x}))$ to discuss new results in this work.
% \kph{In this definition, are we assuming x dominates y? yes }
The volume polynomial is interpreted as an expression that gives the polynomial of the polytope in question by substituting numbers for the $n$ variables in the polynomial.
% For example, if $\mathbf{x} = (x_1)$, then $V_n(\mathbf{x}) = x_1$.
% We can also write 
% \[V_2(\mathbf{x}) =  x_1x_2 + \textstyle\frac{1}{2}x_1^2\]
% and
% \[V_3(\mathbf{x}) = x_1x_2x_3 + \textstyle\frac{1}{2}x_1^2x_2 + \textstyle\frac{1}{2}x_1x_2^2 + \textstyle\frac{1}{2}x_1^2x_3 + \textstyle\frac{1}{6}x_1^3.\]
% \kph{Find source for these polynomials and how to generalize nicely.}
% To generalize this formula, we reference the \cite{bib:PitmanStanley} where the authors prove the following.
Stanley and Pitman \cite{bib:PitmanStanley} explicitly describe this formula which we print in \Cref{thm:volume_formula} for ease of reference.

\begin{theorem}\label{thm:volume_formula} \cite[Theorem 1]{bib:PitmanStanley}
    For each $n = 1, 2, \hdots$,
    \begin{equation}\label{eq:poly_vol_fmla}
    V_n(\mathbf{x}) = \sum_{\mathbf{k} \in K_n} \prod_{i=1}^n \frac{x_i^{k_i}}{k_i!} = \frac{1}{n!} \sum_{\mathbf{k} \in K_n} {\binom{n}{k_1, k_2, \hdots, k_n}} x_1^{k_1}x_2^{k_2} \cdots x_n^{k_n}
    \end{equation}
    where 
    \begin{equation}\label{def:the_set_K_n}
        K_n \coloneqq \left\{ \mathbf{k} \in \mathbb{N}^n : \sum_{i=1}^j k_i \geq j \text{ for all } 1 \leq j \leq n-1 \text{ and } \sum_{i=1}^n k_i = n \right\}
    \end{equation}
    with $\mathbb{N} = \{0, 1, 2, \hdots\}$.
\end{theorem}
We call elements of $K_n$ ``balanced sequences of length $n$.''
In \Cref{sec:appendix_code_for_K_n_short_paper}, we share code used to find and enumerate balanced sequences of length $n$.
This code supports the claim in \cite{bib:PitmanStanley} that $|K_n|$ is Catalan.

\begin{theorem}\label{thm:ps_vol+prkgfns}[\cite{bib:PitmanStanley}, Theorem 11]
    With $P_n(\mathbf{x})$ and $V_n(\mathbf{x})$ as above, we have
    \begin{equation}\label{eq:PS_pf+vol_xpkgfn}
    P_n(\mathbf{x}) 
    =
    \sum_{(a_1, a_2, \hdots, a_n) \in PF_n} x_{a_1}x_{a_2}\cdots x_{a_n} 
    \end{equation}
    and
    \begin{equation}\label{eq:PS_pf+vol_polytope}
    P_n(\mathbf{x}) 
    = 
    n! \cdot V_n(\mathbf{x}).
    \end{equation}
\end{theorem}

The paper is organized as follows.
In \Cref{sec:pf_background_short_paper}, we provided background information related to parking functions and their generalizations.
Parking functions have many applications, and in this paper, we care about their connection to polytope volumes.
Thus, in \Cref{sec:polytope_background_short_paper}, we provided background related to polytopes and, in particular, Pitman-Stanley polytopes.
In \Cref{sec:main_problem_main_result_short_paper}, we present an open problem which aims to connect parking completions and $\mathbf{x}$-parking functions.
Prior work by Adeniran et al.\ \cite{bib:ParkingCompletions} enumerates parking completions, and prior work by Pitman and Stanley \cite{bib:PitmanStanley} enumerates $\mathbf{x}$-parking functions with a distinct expression.
The paper by Adeniran et al.\ suggests that these formulas are related, and the authors were not able to explicitly describe such a relationship.
I present a solution to this problem.
In particular, the proof of \Cref{thm:MainEnumerativeResult} 
gives a bijection between these sets which establishes that these formulas are equivalent, and moreover, the bijection is the identity map.
We close with \Cref{sec:future_work_short_paper} detailing some problems extending this work with solutions to appear in forthcoming papers.
% Afterwards, \Cref{sec:appendix_code_for_K_n_short_paper}
% and
% \Cref{sec:appendix_code_check_all_xpf_short_paper} 
\Cref{sec:appendix_code}
shares SageMath code used
to
% In SageMath (code can be found in \Cref{sec:appendix_code_check_all_xpf_short_paper}), we 
compare the set of parking completions and $\mathbf{x}$-parking functions for all possible cases of $\mathbf{t}$ for $n \in [8]$. The data illustrates and motivated our main result.
% The code timed out for $n = 9$, yet the earlier results were encouraging, and a proof idea was formed.

\section{Main Problem and Main Result}\label{sec:main_problem_main_result_short_paper}

Adeniran et al.\ \cite{bib:ParkingCompletions} connected parking completions to the volumes of Pitman-Stanley polytopes by explicit computation on small values of $n$. 
They prove that the number of parking completions is given by
\begin{equation}\label{eqn:completions_count}
|\PC_n(\mathbf{t})| = \sum_{\mathbf{\ell} \in L_n(\mathbf{t})} \binom{n-m}{\mathbf{\ell}} \prod_{j=1}^{m+1} (\ell_j + 1)^{(\ell_j - 1)}
\end{equation}     
where $L_n (\mathbf{t})=\{\mathbf{\ell} = (\ell_1, \ell_2, \dots, \ell_{m+1} ) \in \mathbb{N}^{(m+1)}$ such that
$\ell_1 + \ell_2 + \dots + \ell_{j} \geq t_j - j$ for all $j \in [m]$ and $ \ell_1 + \ell_2 +  \dots + \ell_{m+1} = n-m$.
Moreover, they note that this agrees with the cardinality of the set of 
$\mathbf{x}$-parking functions which is given by \Cref{eq:PS_pf+vol_polytope}.
They then state that there is no obvious reason why these formulas agree for many $n$.
The authors remark ``there does not seem to be a simple translation between the two.''

We defined $\mathbf{x}$-parking functions earlier in Definition~\ref{def:x_parking_function}.
For ease of later computations, we now define $\mathbf{u}$-parking functions.

\begin{definition}\label{def:u_pkg_fn}
    For $\mathbf{u} \coloneqq (u_1, u_2, \hdots, u_n) \in \mathbb{N}^n$ with $u_1 \leq u_2 \leq \cdots \leq u_n$, define a \textbf{$\mathbf{u}$-parking function} to be a sequence $(a_1, a_2, \hdots, a_n)$ of positive integers whose nondecreasing rearrangement $(a_1^{\uparrow}, a_2^{\uparrow}, \cdots, a_n^{\uparrow})$ satisfies $1 \leq a_i^{\uparrow} \leq u_i$, for each $i\in[n]$.
    Moreover, we let $\PF_n(\mathbf{u})$ denote the set of $\mathbf{u}$-parking functions.
\end{definition} 

Notice that we can convert between these families of parking functions. 
That is, if $\mathbf{a}=(a_1,a_2,\ldots,a_n)$ is a $\mathbf{u}$-parking function for $\mathbf{u} = (u_1, u_2, \hdots, u_n)$, then $\mathbf{a}$ is an $\mathbf{x}$-parking function for $\mathbf{x} = (u_1, u_2-u_1, u_3-u_2, \hdots, u_{n} - u_{n-1})$.
Observe that for all $i \in [n]$, $x_1 + \cdots + x_i = u_i$, so $1 \leq a_i^{\uparrow} \leq u_i$ is equivalent to $1 \leq a_i^{\uparrow} \leq x_1 + \cdots + x_i$.

Additionally, we soon show that there is also a connection between $\mathbf{x}$-parking functions and parking completions. 
In particular, given a vector $\mathbf{x}$, we construct a sequence $\mathbf{t}$ which makes the corresponding sets of generalized parking functions agree. 
This leads to a combinatorial argument that the set of parking completions is the same as the set of $\mathbf{x}$-parking functions.
We formalize this result in \Cref{thm:MainEnumerativeResult}, and to that end, we assemble a series of results to simplify the argument.

Recall that every permutation of an $\mathbf{x}$-parking function is an $\mathbf{x}$-parking function since these are defined using a weakly increasing rearrangement. 
We now prove an analogous result for parking completions.

\begin{remark}
    The definition of parking completions by Adeniran et al.\ \cite{bib:ParkingCompletions} begins by adapting the parking scenario from the classical case to include some spots which are not available (referred to as taken spots), and the authors then describe a definition which concatenates the taken spots and the proposed parking completion. 
    % These notions are equivalent because 
    In this text, we move back and forth between these notions because they are equivalent: naming taken spots means that an initial queue of cars parks in these spots first, and we resolve that in the parking function setting by listing the taken spots at the start of our preference list in increasing order.
    
    Using the parking scenario is often helpful for determining whether something is or is not a parking completion.
    On the other hand, using the concatenation is helpful because it enables us to prove things about parking completions by shifting our perspective to parking functions, for which more results are known. % because the literature on parking functions is expansive.
    We implement this perspective shift to prove Lemma~\ref{lem:any_pc_permutation_works}.
    % In this text, we respect to the parking scenario; here we use the concatenation of t and c. These definitions are equivalent because blah \textcolor{red}{fix this remark}; we can just 
    % we shift to parking function world to get results that work
\end{remark}

\begin{lemma}\label{lem:any_pc_permutation_works}
    Every permutation of a parking completion is a parking completion.
\end{lemma}
\begin{proof}
Let $\mathbf{c} \in \PC_n(\mathbf{t})$.
This implies that all cars with preferences in $\mathbf{c}$ can park after spots in $\mathbf{t}$ are taken.
By \Cref{def:parking_completion}, the concatenation of $\mathbf{t}$ and $\mathbf{c}$, which is denoted $\mathbf{t}|\mathbf{c} \coloneqq (t_1, t_2, \hdots, t_m, c_1, c_2, \hdots, c_{n-m})$, is in the set $\PF_n$.

We invoke Statement~\ref{cor:any_pf_permutation_works} to find that any permutation of $\mathbf{t}|\mathbf{c}$ is a parking function.
Let $\mathbf{c}'$ denote any permutation of $\mathbf{c}$.
Since spots $t_1, t_2, \hdots, t_m$ are distinct, the first $m$ cars park in spots in $\mathbf{t}$.
Permuting entries of $\mathbf{c}$ only affects cars numbered $m+1, m+2, \hdots, n-m$ in the parking order.
Thus, we know $\mathbf{t}|\mathbf{c}' \in \PF_n$.
This implies that $\mathbf{c}' \in \PC_n(\mathbf{t})$ since all cars with preferences in $\mathbf{c}'$ can park after spots in $\mathbf{t}$ are taken.
\end{proof}

Recall that parking functions can be described using an inequality characterization rather than using the parking scenario.
In \cite[Definition~2.1]{bib:ParkingCompletions}, Adeniran et al.\ give an inequality characterization of parking completions which we restate below as a lemma for ease of reference.
%in \Cref{lem:pc_increasing_condition}.
% , and this proof is inspired by the proof of Statement~\ref{prop:pf_increasing_condition}.

\begin{lemma}\label{lem:pc_increasing_condition}
    Fix an increasing list $\mathbf{t} = (t_1, t_2, \ldots, t_m)\in [n]^m$. Let $\mathbf{c} = (c_1, c_2, \ldots, c_{n-m}) \in [n]^{n-m}$, and
    let $\mathbf{c}^\uparrow$ denote the weakly increasing rearrangement of $\mathbf{c}$.
    Then, $\mathbf{c}$ is a parking completion in $\PC_n(\mathbf{t})$ if and only if $c^\uparrow_i \leq u_i$ for all $i \in [n-m]$ where $\mathbf{u} = (u_1, u_2, \ldots, u_{n-m})$ is the increasing rearrangement of $[n] \setminus\{t_1,t_2,\ldots,t_m\}$.
    % $\operatorname{con}(\mathbf{t})$.
\end{lemma}

Using this characterization, we get a different proof of Lemma~\ref{lem:any_pc_permutation_works}.

\begin{proof}[Alternate proof of Lemma~\ref{lem:any_pc_permutation_works}]
    This follows immediately from the inequality description of parking completions in Lemma~\ref{lem:pc_increasing_condition}.
    Every permutation of $\mathbf{c}$ has the same weakly increasing rearrangement.
    By Lemma~\ref{lem:any_pc_permutation_works}, every permutation of $\mathbf{c}$ is a parking completion (since $\mathbf{c}$ itself is a parking completion).
\end{proof}

Notice that we can recover the inequality description defining parking functions (Statement~\eqref{prop:pf_increasing_condition}) from Lemma~\ref{lem:pc_increasing_condition} by setting $\mathbf{t} = ()$ since this implies $u_i = i$ for all $i$.
% Our goal is to describe the correct translation between parking completions and $\mathbf{x}$-parking functions.
% In \Cref{rem:find_x_from_t}, 
Next, we describe the appropriate conditions for $\mathbf{t}$ and $\mathbf{x}$ needed for our proof of \Cref{rem:find_x_from_t}.

\begin{remark}\label{rem:find_x_from_t}
    In what follows, the \textbf{content} of a list $\mathbf{a} = (a_1, a_2, \hdots, a_n)$ is the set $\{a_1,a_2,\ldots,a_n\}$, and we denote it $\operatorname{con}(\mathbf{a})$.
    % {1^{\ell_{1}},2^{\ell_{2}},\ldots,n^{\ell_{n}}\}$, where, for all $1 \leq i \leq n$, $\ell_i$ represents the number of occurrences of $i$ in $\alpha$. 
    % We call $\ell_i$ the \textit{multiplicity} of $i$ in $\alpha$.
    Let $\mathbf{c} = (c_1, c_2, \hdots, c_{n-m}) \in [n]^{n-m}$ be a parking completion of length $n$ for $\mathbf{t} = (t_1, t_2, \hdots, t_m)$ as an increasing subset of $[n]$.
    Let $\mathbf{u} \coloneqq (u_1, u_2, \hdots, u_{n-m})$ denote the increasing rearrangement of $[n] \setminus \operatorname{con}(\mathbf{t})$.
    That is, $\operatorname{con}(\mathbf{t}) \cup \operatorname{con}(\mathbf{u}) = [n]$, $\operatorname{con}(\mathbf{t}) \cap \operatorname{con}(\mathbf{u}) = \varnothing$, and $u_i < u_j$ for all $i < j$.    
    Define 
    \[
    \mathbf{x} \coloneqq (u_1, u_2 - u_1, u_3 - u_2, \hdots, u_{n-m} - u_{(n-m)-1}) \in \mathbb{N}^{N}
    \]
    where $N \coloneqq \operatorname{len}(\mathbf{u}) = \operatorname{len}(\mathbf{x}) = n-m$.
\end{remark}

In what follows, let $\PC_{n}^\uparrow(\mathbf{t})$ denote the set of weakly increasing parking completions of length $n$ for fixed~$\mathbf{t}$.
Similarly, let $\PF_{N}^\uparrow(\mathbf{x})$ denote the set of weakly increasing $\mathbf{x}$-parking functions of length $N$.
We are now ready to prove the main result (\Cref{thm:MainEnumerativeResult}), and we do it in two parts. First, we show that the weakly increasing sets $\PC_{n}^\uparrow(\mathbf{t})$ and $\PF_{n-m}^\uparrow(\mathbf{x})$ (Proposition~\ref{prop:weakly_increasing_sets_are_equal}). 
Then, by the permutation invariance of both sets (Statement~\eqref{cor:any_pf_permutation_works} and Lemma~\ref{lem:any_pc_permutation_works}), we arrive at the desired set equality between $\PC_{n}(\mathbf{t})$ and $\PF_{n-m}(\mathbf{x})$.

% In \Cref{prop:weakly_increasing_sets_are_equal}, we show that $\mathbf{c}$ is an $\mathbf{x}$-parking function for 
% \[
% \mathbf{x} = (u_1, u_2 - u_1, u_3 - u_2, \hdots, u_{n-m} - u_{(n-m)-1}). 
% \]

\begin{proposition}\label{prop:weakly_increasing_sets_are_equal}
    The set of weakly increasing parking completions of length $n$ for a fixed $\mathbf{t}$ is equal to the set of weakly increasing $\mathbf{x}$-parking functions of length $n-m$ with $\mathbf{x}$ obtained from $\mathbf{t}$ as described in Remark~\ref{rem:find_x_from_t}.
    Namely, $\PC_{n}^\uparrow(\mathbf{t}) = \PF_{n-m}^\uparrow(\mathbf{x})$ with $\mathbf{t}$ and $\mathbf{x}$ as in Remark~\ref{rem:find_x_from_t}.
\end{proposition}

\begin{proof}
    We proceed via a double containment argument.

    Let $\mathbf{c} \in \PC_n^\uparrow(\mathbf{t})$ for $\mathbf{t} = (t_1, t_2, \hdots, t_m)$.
    We show that $\mathbf{c}^\uparrow \in \PF_{n-m}^\uparrow(\mathbf{x})$.
    Equivalently, we show that $\mathbf{c}^\uparrow$ satisfies $c_i^\uparrow \leq x_1 + x_2 + \cdots + x_i$ for all $i \in [n-m]$.
    Notice $\mathbf{c} = \mathbf{c}^\uparrow$ because we assumed that $\mathbf{c}$ was weakly increasing.
    Now, observe $c_1 \leq x_1 = u_1$ because $u_1$ is the first unoccupied spot on the street in the parking completion setting.
    If $c_1 > u_1 = x_1$, then spot $u_1$ is left empty, and $\mathbf{c} \notin \PC_{n}^\uparrow(\mathbf{t})$ which is a contradiction.
    This is the base case of our inductive argument.
    
    For the induction hypothesis, we assume that if $\mathbf{c} = (c_1, c_2, \hdots, c_k) \in \PC_{m+k}^\uparrow(\mathbf{t})$, then $\mathbf{c} \in \PF_{k}^\uparrow(\mathbf{x})$.
    For the induction step, consider $\mathbf{d} = (d_1, d_2, \hdots, d_k, d_{k+1}) \in \PC_{m+(k+1)}^\uparrow(\mathbf{t})$.
    We show $\mathbf{d} \in \PF_{k+1}^\uparrow(\mathbf{x})$.
    Using the induction hypothesis, $(d_1, d_2, \hdots, d_k) \in \PF_{k}^\uparrow(\mathbf{x})$, and thus, $d_i \leq x_1 + x_2 + \cdots + x_i$ for $i \in [k]$ by definition of $\mathbf{x}$-parking function.
    It remains to check that $d_{k+1} \leq x_1 + x_2 + \cdots + x_{k+1}$.
    Substituting $x_1 = u_1$ and $x_i = u_{i} - u_{i-1}$ for $i \in \{2, 3, \ldots, k+1\}$, we have a telescoping sum that simplifies to show $x_1 + x_2 + \cdots + x_{k+1} = u_{k+1}$.
    We immediately have that $d_{k+1} \leq u_{k+1} = x_1 + x_2 + \cdots + x_{k+1}$ by the inequality characterization of parking completions.
    This shows $\PC_{n}^\uparrow(\mathbf{x}) \subseteq \PF_{n-m}^\uparrow(\mathbf{x})$.

    For the other containment, let $\mathbf{a} = (a_1, a_2, \hdots, a_{n-m}) \in \PF_{n-m}^\uparrow(\mathbf{x})$.
    We show $\mathbf{a} \in \PC_{m+(n-m)}^\uparrow(\mathbf{t})$.
    Equivalently, we show that $\mathbf{a}^\uparrow$ satisfies $a_i^{\uparrow} \leq u_i$ for all $i \in [n]$.
    Notice $\mathbf{a} = \mathbf{a}^\uparrow$ because we assumed that $\mathbf{a}$ was weakly increasing.
    Now, observe $a_1 \leq x_1 = u_1$ by Definition~\ref{def:x_parking_function} and Remark~\ref{rem:find_x_from_t}.
    By Remark~\ref{rem:find_x_from_t}, for all $i \in \{2, 3, \hdots, n-m\}$, we have that $x_i = u_i - u_{i-1}$.
    This implies $a_i \leq x_1 + x_2 + \hdots + x_i = u_i$ for all $i \in \{ 2, 3, \hdots, n-m \}$, as desired.
    We can obtain $\mathbf{u} = (u_1, u_2, \hdots, u_{n-m})$ using the relationships mentioned in the previous two sentences, and we can obtain $\mathbf{t} = (t_1, t_2, \hdots, t_m)$ as the increasing rearrangement of $[n] \setminus \operatorname{con}(\mathbf{u})$, and in this way, $\mathbf{a}$ is a parking completion of length $n$ for $\mathbf{t}$.
    This shows $\PF_{n-m}^\uparrow(\mathbf{x}) \subseteq \PC_{n}^\uparrow(\mathbf{x})$.

    We have shown double containment, and this concludes the proof.
\end{proof}

\begin{theorem}\label{thm:MainEnumerativeResult}
    Let $\mathbf{t} = (t_1, t_2, \hdots, t_m)$.
    % u = (u1, u2, dots, un-m) where content (t) \cap content(u) empty
    Then, $\mathbf{x} = (u_1, u_2 - u_1, u_3 - u_2, \hdots, u_{n-m} - u_{(n-m)-1})$ using Remark~\ref{rem:find_x_from_t}.
    In this way, the set of parking completions of length $n$ for $\mathbf{t}$ is exactly the set of $\mathbf{x}$-parking functions of length $n-m$.
    Namely, $\PC_{n}(\mathbf{t}) = \PF_{n-m}(\mathbf{x})$ with $\mathbf{t}$ and $\mathbf{x}$ as in Remark~\ref{rem:find_x_from_t}.
\end{theorem}

\begin{proof}
    By \Cref{prop:weakly_increasing_sets_are_equal}, we know the weakly increasing versions of the sets in question agree.
    Since parking completions are permutation invariant by \Cref{lem:any_pc_permutation_works} and since $\mathbf{x}$-parking functions are permutation invariant by Definition~\ref{def:x_parking_function}, we can permute the weakly increasing versions of these sets respectively to obtain the final result.
\end{proof}

% \begin{corollary}
%     With $\mathbf{t}$ and $\mathbf{x}$ as in \Cref{thm:MainEnumerativeResult},
%     % Let $\mathbf{t} = (t_1, t_2, \hdots, t_m)$.
%     % % u = (u1, u2, dots, un-m) where content (t) \cap content(u) empty
%     % Then, $\mathbf{x} = (u_1, u_2 - u_1, u_3 - u_2, \hdots, u_{n-m} - u_{(n-m)-1})$ using \Cref{rem:find_x_from_t}.
%     there exists a bijection between the set of parking completions of length $n$ for $\mathbf{t}$ and the set of $\mathbf{x}$-parking functions of length $n-m$.
% \end{corollary}

% \begin{proof}
%     The proof follows immediately from \Cref{thm:MainEnumerativeResult}, with the bijection being the identity map.
% \end{proof}

\begin{corollary}
    With $\mathbf{t}$ and $\mathbf{x}$ as in \Cref{thm:MainEnumerativeResult}, we have
    $
    |\PC_n(\mathbf{t})| = P_n(\mathbf{x})
    $.
\end{corollary}
\begin{proof}
    This follows immediately since the sets of objects are identical.
\end{proof}

\section{Future Work}\label{sec:future_work_short_paper}
There are three immediate ways this work can be extended.
\begin{enumerate}
    \item 
    In \cite[Section 5]{bib:ParkingCompletions}, Adeniran et al.\ illustrate that \Cref{eqn:completions_count} and \Cref{eq:PS_pf+vol_polytope} often return different numbers of terms.
    Computationally, it appears \Cref{eqn:completions_count} has more terms about half of the time, and \Cref{eq:PS_pf+vol_polytope} has more terms a little less than half of the time. 
    However, it remains an open problem to determine conditions on $\mathbf{t}$ for which one formula is computationally less expensive than the other. 
    
    \item Given the previous question, it is of interest to give a purely algebraic manipulation that converts \Cref{eqn:completions_count} into \Cref{eq:PS_pf+vol_polytope}. 
    If not possible in full generality, we ask if this can be done in some specific cases.

    \item In \cite{bib:PitmanStanley}, the authors enumerate $\mathbf{x}$-parking functions as the volume of a polytope.
    Knowing now that these objects are parking completions, we want to know whether there is a polytope whose volume enumerates parking completions. 
    This would give a geometric interpretation for \Cref{eqn:completions_count}. 
    % (This problem is inspired by the fact that there does exist a geometric interpretation for \Cref{eq:PS_pf+vol_polytope}.)
\end{enumerate}

\newpage
\bibliographystyle{amsplain}
  \bibliography{2bibliography}

% \bibliographystyle{plain}

% \newpage
\appendix
\section{Code}\label{sec:appendix_code}
\subsection{Balanced sequences of length $n$}\label{sec:appendix_code_for_K_n_short_paper}
The code defines a function which, given $n$, produces the elements of $K_n$ as well as enumerates them.
\begin{lstlisting}[language=Python]
    def K(n):
    IVn = IntegerVectors(n,n).list()
    K = []
    for v in IVn:
        for i in range(n):
            if sum(v[:i]) >= i:
                if i == n-1:
                    K.append(v)
            else:
                break
    print(K)
    print(len(K))
\end{lstlisting}

\subsection{Comparing parking completions and $\mathbf{x}$-parking functions}\label{sec:appendix_code_check_all_xpf_short_paper}
This code defines a function which, given $n$, ranges over all possible $\mathbf{t}$ and computes all corresponding parking completions.
The function simultaneously computes all $\mathbf{x}$-parking functions corresponding to each choice of $\mathbf{t}$, and returns whether these sets have all of the same elements.
(Note we do call functions that are not described in this particular excerpt of code.)

\begin{lstlisting}[language=Python]
import time

start = time.time()
n=4
win = true
for t in find_all_t(n):
    x = buildX(n, t)
    if not compare(n,t,x):
        print(f"NO FREAKING WAY t={t}\nx={x}")
        win = false
if win:
    print("YOU WIN LEVEL", n,"!")
end = time.time()

print("Runtime:", end - start, "seconds")
\end{lstlisting}

\end{document}